\documentclass{amsart}
\usepackage{amsthm,amssymb,amsmath}
\usepackage{epsfig}
\usepackage{color}

\newtheorem{prelem}{{\bf Proposition}}

\newtheorem{theorem}{Theorem}
\newtheorem{corollary}[theorem]{Corollary}

\newtheorem{observation}[theorem]{Observation}
\newtheorem{proposition}[theorem]{Proposition}

\theoremstyle{definition}

\theoremstyle{remark}

\input{tcilatex}

\begin{document}
\title{$k$-tuple total restrained domination and\\
$k$-tuple total restrained domatic in\\
graphs}
\author{ Adel P. Kazemi \vspace{4mm}\\
Department of Mathematics\\
University of Mohaghegh Ardabili\\
P. O. Box 5619911367, Ardabil, Iran\\
adelpkazemi@yahoo.com\vspace{3mm} \\
}
\date{}
\maketitle

\begin{abstract}
Let $G$ be a graph of order $n$ and size $m$ and let $k\geq 1$ be an
integer. A $k$-tuple total dominating set in $G$ is called a
$k$-tuple total restrained dominating set of $G$ if each vertex
$x\in V(G)-S$ is adjacent to at least $k$ vertices of $V(G)-S$. The
minimum number of vertices of a such sets in $G$ are the $k$-tuple
total restrained domination number $\gamma _{\times k,t}^{r}(G)$ of
$G$. The maximum number of classes of a partition of $V(G)$ such
that its all classes are $k$-tuple total restrained dominating sets
in $G$, is called the $k$-tuple total restrained domatic number of
$G$.

In this manuscript, we first find $\gamma _{\times k,t}^{r}(G)$,
when $G$ is complete graph, cycle, bipartite graph and the
complement of path or cycle. Also we will find bounds for this
number when $G$ is a complete multipartite graph. Then we will know
the structure of graphs $G$ which $\gamma _{\times k,t}^{r}(G)=m$,
for some $m\geq k+1$ and give upper and lower bounds for $\gamma
_{\times k,t}^{r}(G)$, when $G$ is an arbitrary graph. Next, we
mainly present basic properties of the $k$-tuple total restrained
domatic number of a graph and give bounds for it. Finally we give
bounds for the $k$-tuple total restrained domination number of the
complementary prism $G\overline {G}$ in terms on the similar number
of $G$ and $\overline {G}$ when $G$ is a regular graph or an
arbitrary graph. And then we calculate it when $G$ is cycle or path.
\end{abstract}



%%%%%%%%%%%%%%%%%%%%%%%%%%%%%%%%%%%%%%%%%%%%%%%%%%%%%% Introduction %%%%%%%%%%%%%%%%%%%%%%%%%%%%%%%%%%%%%%%%%%%%%%%%%%%%%%%%%%%%%

\section{Introduction}

The research of the domination in graphs has been an evergreen of
the graph theory. Its basic concept is the dominating set and the
domination number. A numerical invariant of a graph which is in a
certain sense dual to it is the domatic number of a graph. And many
variants of the dominating set were introduced and the corresponding
numerical invariants were defined for them. Here, we initial to
study of the $k$-tuple total restrained domination number and the
$k$-tuple total restrained domatic number.

We start with definitions of various concepts concerning the
domination in graphs. A subset $S\subseteq V(G)$ is called a
$k$-\emph{tuple total dominating set}, briefly kTDS, \cite{HK} in
$G$, if for each $x\in V(G)$, $\mid N(x)\cap S\mid \geq k$. Recall
that $1$-tuple total dominating set is known as \emph{total
dominating set}.

Let $k\geq 1$ be an integer. A $k$-tuple total dominating set in $G$
is called a $k$-\emph{tuple total restrained dominating set},
briefly kTRDS, in $G$, if each vertex $x\in V(G)-S$ is adjacent to
at least $k$ vertices of $V(G)-S$. The minimum number of vertices of
a $k$-tuple total dominating set in a graph $G$ is the
$k$-\emph{tuple total domination number} of $G$ and denoted by
$\gamma _{\times k,t}(G)$. Analogously the $k$-\emph{tuple total
restrained domination number} $\gamma _{\times k,t}^{r}(G)$ is
defined. Obviously, $\gamma _{\times k,t}(G)\leq \gamma _{\times
k,t}^{r}(G)$.

The domatic number of a graph was introduced in \cite{CH}, and the
total domatic number in \cite{CDH}. Sheikholeslami and Volkmann
extended the last definition to the $k$-tuple total domatic number
$d_{\times k,t}(G)$ in \cite{SV}. In an analogous way we will define
the $k$-tuple total restrained domatic number and then we will
discuss the purpose of defining it. Let $\mathbf{D}
$ be a partition of the vertex set $V(G)$ of $G$. If all classes of $\mathbf{%
D}$ are $k$-tuple total restrained dominating sets in $G$, then
$\mathbf{D}$ is called a $k$-\emph{tuple total restrained domatic
partition}, briefly kTRDP, of $G$. The maximum number of classes of
a $k$-tuple total restrained domatic partition of $G$ is the
$k$-\emph{tuple total restrained domatic number} $d_{\times
k,t}^r(G)$ of $G$.

Haynes et al. in \cite{HaHeSl07} defined a new type of graph product
that generalizes the concept of a cartesian product. Let $G$ and $H$
be two graphs with the vertices sets $V(G)=\{u_{i}\mid 1\leq i\leq
n\}$ and $V(H)=\{v_{j}\mid 1\leq j\leq p\}$. Let $R$ be a subset of
$V(G)$ and $S$ be a subset of $V(H)$. The \emph{complementary product}
$G(R)\square H(S)$ are defined as follows. The vertex set $G(R)\square H(S)$ is $\{(u_{i},v_{j}):1%
\leq i\leq n,1\leq j\leq p\}$. And the edge $(u_{i},v_{j})(u_{h},v_{k})$ is
in $E(G(R)\square H(S))$

1. if $i=h$, $u_{i}\in R$ and $v_{j}v_{k}\in E(H)$, or if $i=h$, $%
u_{i}\notin R$ and $v_{j}v_{k}\notin E(H)$, or

2. if $j=k$, $v_{j}\in S$ and $u_{i}u_{h}\in E(G)$, or if $j=k$, $%
v_{j}\notin S$ and $u_{i}u_{h}\notin E(G)$.

In other words, for each $u_{i}\in V(G)$, we replace $u_{i}$ with a copy of $%
H$ if $u_{i}$ is in $R$ and with a copy of its complement $\overline{H}$ if $%
u_{i}$ is not in $R$, and for each $v_{j}\in V(H)$, we replace each $v_{j}$
with a copy of $G$ if $v_{j}\in S$ and a copy of $\overline{G}$ if $%
v_{j}\notin S$. If $R=V(G)$ (respectively, $S=V(H)$), we write simply $%
G\square H(S)$ (respectively, $G(R)\square H$). Thus, $G\square H(S)$ is the
graph obtained by replacing each vertex $v$ of $H$ by a copy of $G$ if $v\in
S$ and by a copy of $\overline{G}$ if $v\notin S$, and replacing each $u_{i}$
with a copy of $H$. Therefore, the cartesian product of $G$ and $H$ is
simply $G(V(G))\square H(V(H))=G\square H $.

The \emph{complementary prism} $G\overline{G}$ of a graph $G$ is the
special complementary product $G\square K_{2}(S)$\ where $\mid S\mid
=1$. In other words the complementary prism $G\overline{G}$ of $G$
is the graph formed from the disjoint union $G\cup \overline{G}$ of
$G$ and $\overline{G}$ by adding the edges of a perfect matching
between the corresponding vertices
(same label) of $G$ and $\overline{G}$. For example, the graph $C_{5}%
\overline{C_{5}}$ is the Petersen graph. Also, if $G=K_{n}$, the graph $K_{n}%
\overline{K_{n}}$ is the corona $K_{n}\circ K_{1}$, where the
\emph{corona} $G\circ K_{1}$ of a graph $G$ is the graph obtained
from $G$ by attaching a pendant edge to each vertex of $G$.

The $k$-\emph{join} $G\circ _{k}H$ of a graph $G$ to a graph $H$ of
order at least~$k$ is the graph obtained from the disjoint union of
$G$ and $H$ by joining each vertex of $G$ to at least~$k$ vertices
of $H$.

% Also for any two graphs $G$ and $H$ the $corona$ of $H$ and $G$
%is the graph $H\circ G$\
%formed from one copy of $H$ and $\mid V(H)\mid $\ copies of $G$\ where the $%
%ith$ vertex of $H$\ is adjacent to every vertex in the $ith$ copy of
%$G$.

The notation we use is as follows. Let $G$ be a simple graph with
\emph{vertex set} $V=V(G)$ and \emph{edge set} $E=E(G)$. The
\emph{order} $\mid V\mid $ and \emph{size} $\mid E\mid $ of $G$ are
respectively denoted by $n=n(G)$ and $m=m(G)$. For
every vertex $v\in V$, the \emph{open neighborhood} $N_{G}(v)$ is the set $%
\{u\in V\mid uv\in E\}$ and its \emph{closed neighborhood} is the set $%
N_{G}[v]=N_{G}(v)\cup \{v\}$. The \emph{degree} of a vertex $v\in V$ is $%
deg(v)=\mid N(v)\mid $. The \emph{minimum} and \emph{maximum degree}
of a graph $G$ are denoted by $\delta =\delta (G)$ and $\Delta
=\Delta (G)$, respectively. If every vertex of $G$ has degree $k$,
then $G$ is said to be $k$-\emph{regular}. The \emph{complement} of
a graph $G$ is denoted by $\overline{G}$ which is a
graph with $V(\overline{G})=V(G)$ and for every two vertices $v$ and $w$, $%
vw\in E(\overline{G})$\ if and only if $vw\notin E(G)$. We write
$K_{n}$ for the \emph{complete graph} of order $n$ and
$K_{n_{1},...,n_{p}}$ for the \emph{complete $p$-partite graph}.

Also we write $C_{n}$ and $P_{n}$, respectively, for a \emph{cycle}
and a \emph{path} of order $n$, in which $C_n\overline{C_n}$ and
$P_n\overline{P_n}$ denote their complementary prisms. Here we
assume that $V(C_n)=V(P_n)=\{i\mid 1\leq i\leq n\}$ and
$E(C_{n})=E(P_{n})\cup \{1n\}=\{ij\mid 1\leq i<j\leq n \mbox{ and
}j\equiv i+1 \mbox{ (mod } n) \}$. We also assume
$V(\overline{G})=\{\overline{i}\mid 1\leq i\leq n\}$, where $G$ is
$C_n$ or $P_n$, and every vertex $i$ in $G$ is adjacent to its
respective vertex $\overline {i}$ in $\overline{G}$.

This paper is organized as follows. In section 2, we present the
$k$-tuple total restrained domination number of the complete graphs,
cycles, bipartite graphs and the complement of paths or cycles. Also
we will present some bounds for the $k$-tuple total restrained
domination number of the complete multipartite graph. Then, in
section 3, we will show the structure of graphs $G$ which $\gamma
_{\times k,t}^{r}(G)=m$, for some $m\geq k+1$ and give upper and
lower bounds for $\gamma _{\times k,t}^{r}(G)$, when $G$ is an
arbitrary graph. In the next section, we mainly present basic
properties of the $k$-tuple total restrained domatic number of a
graph and give bounds for it. Also we give some sufficient
conditions for the $k$-tuple domination (resp. domatic) number of a
graph is its $k$-tuple restrained domination (resp. domatic) number.
Finally, in the last section, we give some bounds for the $k$-tuple
total restrained domination number of the complementary prism
$G\overline {G}$ in terms on the similar number of $G$ and
$\overline {G}$ when $G$ is a regular graph or an arbitrary graph.
And then we calculate it for the complementary prism of a cycle or
path.

The following observations and propositions are useful.

\begin{observation}
\label{P2} Let $G$ be a graph of order $n$ in which $\delta (G)\geq
k$. Then

\textbf{i.} every vertex of degree at most $2k-1$ of $G$ and at
least its $k$ neighbors belong to every kTRDS,

\textbf{ii.} if $\delta (G)\leq 2k-1$, then $d_{\times
k,t}^{r}(G)=1$,

\textbf{iii.}  if $\gamma _{\times k,t}^{r}(G)<n$, then $\Delta
(G)\geq 2k$, and so $n\geq 2k+2$.
\end{observation}

\begin{observation}
\label{Ob.comp.Graph} Let $k<n$ be two positive integers. Then
$d_{\times k,t}^r(K_n)=\lfloor \frac{n}{k+1} \rfloor$.
\end{observation}

%\begin{prelem} \label{P4.Kaz} \emph{(\textbf{Kazemi \cite{Kaz}
%2011})} Let $G$ be a graph of order $n$ with $\delta (G)\geq k\geq
%1$.

%\textbf{i. }If every two vertices of $G$ have no neighbor in common or have
%at least $k$ neighbors in common, then
%\begin{equation*}
%\gamma _{\times k,t}(G)\leq \left( \frac{1+\ln \delta }{\delta }\right)
%\left( _{k}^{n}\right) ,
%\end{equation*}

%\textbf{ii. }if $\delta >k$ and $\overset{\curlywedge }{d_{k}}=\frac{1}{n}%
%\sum\limits_{v\in V(G)}\left( _{k}^{\deg (v)}\right) ,$ then
%\begin{equation*}
%\gamma _{\times k,t}(G)\leq \left( \frac{\ln (\delta -k)+\ln \overset{%
%\curlywedge }{d_{k}}+1}{\delta -k}\right) n.
%\end{equation*}
%\end{prelem}

\begin{prelem}
\label{TCnCn} \emph{(\textbf{Kazemi \cite{Kaz2} 2011})} Let $n\geq 4$. Then%
\begin{equation*}
\gamma _{t}(C_{n}\overline{C_{n}})=\left\{
\begin{array}{ll}
2\left\lceil n/4\right\rceil +2 & \mbox{if }n\equiv 0\mbox{ (mod
}4\mbox{)}, \\
2\left\lceil n/4\right\rceil +1 & \mbox{if }n\equiv 3\mbox{ (mod
}4\mbox{)}, \\
2\left\lceil n/4\right\rceil & \mbox{Otherwise. }%
\end{array}%
\right.
\end{equation*}
\end{prelem}

\begin{prelem}
\label{DCnCn} \emph{(\textbf{Kazemi \cite{Kaz2} 2011})} If $n\geq 5$, then $\gamma _{\times 2,t}(C_{n}\overline{C_{n}}%
)=n+2.$
\end{prelem}

\begin{prelem}
\label{TPnPn} \emph{(\textbf{Kazemi \cite{Kaz2} 2011})} Let $n\geq 4$. Then%
\begin{equation*}
\gamma _{t}(P_{n}\overline{P_{n}})=\left\{
\begin{array}{ll}
2\left\lceil (n-2)/4\right\rceil +1 & \mbox{if }n\equiv 3\mbox{(}
\mbox{ mod
}4\mbox{)}, \\
2\left\lceil (n-2)/4\right\rceil +2 & \mbox{otherwise.}%
\end{array}%
\right.
\end{equation*}
\end{prelem}

%\begin{prelem}
%\label{SV1} \emph{(\textbf{Sheikholeslami, Volkmann \cite{SV}
%2011})} For every graph $G$ with $\delta (G)\geq k$,
%\begin{equation*}
%d_{\times k,t}(G)\leq \delta (G)/k.
%\end{equation*}
%\end{prelem}

%\begin{prelem}
%\label{SV2} \emph{(\textbf{Sheikholeslami, Volkmann \cite{SV}
%2011})} For every graph $G$ of order $n$ in which $min\{\delta (G),
%\delta (\overline{G})\}\geq k$,
%\begin{equation*}
%d_{\times k,t}(G) + \ d_{\times k,t}(\overline{G})  \leq (n-1)/k.
%\end{equation*}
%\end{prelem}

%%%%%%%%%%%%%%%%%%%%%%%%%%%%%%%%%%%%%%%%%%%%%%%%%%%%%% $k$-tuple total restrained domination number %%%%%%%%%%%%%%%%%%%%%%%%%%%%%%%%%%%%%%%%%%%%%%%%%%%%%%%%%%%%%

\section{$k$-tuple total restrained domination number in some graphs}

By Observation \ref{P2}($iii$), we have $\gamma _{\times
k,t}^{r}(K_{n})=n$ if $n\leq 2k+1$. Since also every $(k+1)$-subset
of vertices is a kTRDS of $K_n$, when $n\geq 2k+2$, then we have the
next result.

\begin{proposition}
\label{Kn1} Let $k<n$ be positive integers. Then
\begin{equation*}
\gamma _{\times k,t}^{r}(K_n)=\left\{
\begin{array}{cc}
n & \mbox{if }n\leq 2k+1, \\
k+1 & \mbox{otherwise.}%
\end{array}
\right.
\end{equation*}
\end{proposition}

Next three propositions present $\gamma _{\times
k,t}^{r}(\overline{C_{n}})$, $\gamma _{\times
k,t}^{r}(\overline{P_{n}})$ and $\gamma _{\times k,t}^{r}(C_{n})$.

\begin{proposition}
\label{kTRCn-} Let $n\geq k+3\geq 4$. Then%
\begin{equation*}
\gamma _{\times k,t}^{r}(\overline{C_{n}})=\left\{
\begin{array}{ll}
n & \mbox{if }n\leq 2k+2, \\
k+2 & \mbox{if }2k+3 \leq n\leq 3k+2, \\
k+1 & \mbox{if }n\geq 3k+3.%
\end{array}%
\right.
\end{equation*}
\end{proposition}

\begin{proof}
We first prove that $\gamma _{\times k,t}^{r}(\overline{C_{n}})=k+1$
if and only if $n\geq 3k+3$. Let $S$ be a kTRDS of
$\overline{C_{n}}$ with cardinal $k+1$. Then for every two arbitrary
vertices $\overline i$ and $\overline j$ in $S$, $\mid \overline {i}
- \overline j \mid \geq 3$, and so $n\geq 3k+3$. Since also
$\{\overline {3i+1} \mid 0\leq i\leq k\}$ is a kTRDS of
$\overline{C_{n}}$, when $n\geq 3k+3$, then $\gamma _{\times
k,t}^{r}(\overline{C_{n}})=k+1$.

Observation \ref{P2}($iii$) follows that $\gamma _{\times
k,t}^{r}(\overline{C_{n}})=n$ if and only if $k+3\leq n\leq 2k+1$.
Let now $n=2k+2$. Then $\delta (G)=\Delta (G)=n-3=2k-1$. Let $S$ be
a kTRDS of $\overline{C_{n}}$. Let $\overline i \in V-S$. Since
$\mid N(\overline i)\cap S\mid\geq k$ and $\mid N(\overline i)\cap
(V-S)\mid\geq k$, then $deg(\overline i)\geq 2k$ that is not
possible. Therefore $S=V$ and so $\gamma _{\times
k,t}^{r}(\overline{C_{n}})=n$. For the other cases, obviously
$S=\{\overline{2i + 1} \mid 0\leq i\leq k+1\}$ is a kTRDS of
$\overline{C_{n}}$ and so $\gamma _{\times
k,t}^{r}(\overline{C_{n}})=k+2$.
\end{proof}

\begin{proposition}
\label{kTRPn-} Let $n\geq k+3\geq 4$. Then
\begin{equation*}
\gamma _{t}^{r}(\overline{P_{n}})=\left\{
\begin{array}{ll}
n & \mbox{if }n=4, \\
2 & \mbox{if }n\geq 5,%
\end{array}%
\right.
\end{equation*}
and if $k\geq 2$, then
\begin{equation*}
\gamma _{\times k,t}^{r}(\overline{P_{n}})=\left\{
\begin{array}{ll}
n & \mbox{if }n\leq 2k+2, \\
k+2 & \mbox{if }2k+3 \leq n\leq 3k, \\
k+1 & \mbox{if }n\geq 3k+1.%
\end{array}%
\right.
\end{equation*}
\end{proposition}

\begin{proof}
One can verify that $\gamma _{t}^{r}(\overline{P_{n}})$ is $2$ if
and only if $n\geq 5$, and otherwise is $n$. Let now $k\geq 2$. It
can be easily verify that $\gamma _{\times
k,t}^{r}(\overline{P_{n}})=k+1$ if and only if there exists a kTRDS
$S$ of $\overline{P_{n}}$ such that for every two disjoint vertices
$\overline i$ and $\overline j$ in $S$, the difference between
$\overline {i}$ and $\overline j$ to modulo $n$ is at least 3 or
$\{\overline {i},\overline j\}=\{\overline {1},\overline n\}$. And
this is equivalent to $n\geq 3k+1$. Since $S=\{\overline{3i+1} \mid
0\leq i\leq k-1\}\cup \{\overline n\}$ is a kTRDS of
$\overline{P_{n}}$, for $n\geq 3k+1$, then $\gamma _{\times
k,t}^{r}(\overline{P_{n}})=k+1$. Let now $n=k+i\leq 3k$, and let $S$
be a kTRDS of $\overline{P_{n}}$. For every vertex $x$ in $V-S$,
$deg(x)\geq n-1-\mid S \mid \geq n-k-3=i-3$. Since also $deg(x)\geq
k$, then $i\geq k+3$. Hence $\gamma _{\times
k,t}^{r}(\overline{P_{n}})=n$ if $n\leq 2k+2$. Let now $2k+3\leq
n\leq 3k$. Since $\overline {C_n}$ is a spanning subgraph of
$\overline {P_n}$, then $\gamma _{\times
k,t}^{r}(\overline{P_{n}})\leq \gamma _{\times
k,t}^{r}(\overline{C_{n}})=k+2$, by Proposition \ref{kTRCn-}. Now
$\gamma _{\times k,t}^{r}(\overline{P_{n}})>k+1$ follows $\gamma
_{\times k,t}^{r}(\overline{P_{n}})=k+2$.
\end{proof}

\begin{proposition}
\label{TRCn} Let $n\geq 4$. Then $\gamma _{\times 2,t}^{r}(C_n)=n$ and%
\begin{equation*}
\gamma _{t}^{r}(C_{n})=\left\{
\begin{array}{ll}
2\left\lceil n/4\right\rceil -1 & \mbox{if }n\equiv 1\mbox{ (mod
}4\mbox{)}, \\
2\left\lceil n/4\right\rceil +1 & \mbox{if }n\equiv 3\mbox{ (mod
}4\mbox{)}, \\
2\left\lceil n/4\right\rceil & \mbox{Otherwise. }%
\end{array}%
\right.
\end{equation*}
\end{proposition}

\begin{proof}
It is trivial that $\gamma _{\times 2,t}^{r}(C_n)=n$. We note that
\begin{equation*}
\gamma _{t}(C_{n})=\left\{
\begin{array}{ll}
2\left\lceil n/4\right\rceil -1 & \mbox{if }n\equiv 1\mbox{ (mod
}4\mbox{)}, \\
2\left\lceil n/4\right\rceil & \mbox{Otherwise. }%
\end{array}%
\right.
\end{equation*}%
If $n\equiv 0,1,2\mbox{ (mod }4\mbox{)}$, since the corresponding
sets $S_{0}=\{2+4i,3+4i\mid 0\leq i\leq \left\lfloor
n/4\right\rfloor -1\}$, $S_{1}=S_{0}\cup \{n-1\}$ and
$S_{2}=S_{0}\cup \{1,n-2\}$\ are total restrained dominating sets
with
cardinal $\gamma _{t}(C_{n})$, then we have proved proposition, when $%
n\not\equiv 3\mbox{ (mod }4\mbox{)}$. Let now $n\equiv 3\mbox{ (mod
}4\mbox{)}$. Then it can be easily verify that $\gamma
_{t}^{r}(C_{n})\geq \gamma _{t}(C_{n})+1$, and since
$S_{3}=S_{0}\cup \{1,n-3,n\}$\ is a total
restrained dominating set of $C_n$ with cardinal $\gamma _{t}(C_{n})+1$, then $%
\gamma _{t}^{r}(C_{n})=2\left\lceil n/4\right\rceil +1$.
\end{proof}

Now we present the $k$-tuple total restrained domination number of
the bipartite graphs.

\begin{proposition}
\label{bipartite} Let $G$ be a bipartite graph with $\delta (G)\geq
k\geq 1$. Then $2k\leq \gamma _{\times k,t}^r(G)\leq n$. Moreover,
if $X$ and $Y$ are the bipartite sets of $G$, then $\gamma _{\times
k,t}^r(G)=2k$ if and only if there exist two $k$-subsets $S\subseteq
X$ and $T\subseteq Y$ such that for each vertex $x\in X$,
$N(x)\supseteq T$, and for each vertex $y\in Y$, $N(y)\supseteq S$
and the minimum degree of the induced subgraph $G[(X-S)\cup (Y-T)]$
is at least $k$.
\end{proposition}

\begin{proof}
Let $D$ be a $\gamma_{\times k,t}^r(G)$-set, and let $w\in X$ and
$z\in Y$ be two arbitrary vertices. The definition implies that
$\mid D\cap N(w)\mid\geq k$ and $\mid D\cap N(z)\mid\geq k$. Since
$N(w) \cap N(z)=\emptyset$, we deduce that $\mid D\mid \geq 2k$ and
thus $2k \leq \gamma _{\times k,t}^r(G)\leq n$. If there exist two
$k$-subsets $S\subseteq X$ and $T\subseteq Y$ such that for each
vertex $x\in X$, $N(x)\supseteq T$, and for each vertex $y\in Y$,
$N(y)\supseteq S$ and also the minimum degree of the induced
subgraph $G[(X-S)\cup (Y-T)]$ is at least $k$, then obviously
$D=S\cup T$ is a $k$-tuple total restrained dominating set of $G$.
This implies $\gamma _{\times k,t}^r(G)\leq 2k$ and so $\gamma
_{\times k,t}^r(G)=2k$.

Conversely, assume that $\gamma _{\times k,t}^r(G)=2k$, and let $D$
be a $\gamma _{\times k,t}^r(G)$-set. It follows that
\begin{equation*}
\mid D\cap X\mid=\mid D\cap Y\mid=k.
\end{equation*}
Now let $S=D\cap X$ and $T=D\cap Y$.  Then $T\subseteq N(x)$ for
each vertex $x\in X$ and $S\subseteq N(y)$ for each vertex $y\in Y$.
Now if $\mid X\mid >k$ and $\mid Y\mid >k$, then, by the definition,
$\delta (G[(X-S)\cup (Y-T)])\geq k$ and the proof is complete.
\end{proof}

\begin{corollary}
\label{comp.bipartite} Let $G=K_{n,m}$ be a complete bipartite graph
with $n\geq m \geq k\geq 1$. Then
\begin{equation*}
\gamma _{\times k,t}^{r}(G)=\left\{
\begin{array}{ll}
2k & \mbox{if }n\geq m\geq 2k, \\
n+m & \mbox{otherwise. }
\end{array}%
\right.
\end{equation*}
\end{corollary}

Now we present some bounds for $\gamma _{\times k,t}^{r}(G)$, where
$G=K_{n_{1},...,n_{p}}$ is a complete multipartite graph and $p\geq
3$.

\begin{proposition}
\label{Kn,...,m1} Let $G=K_{n_{1},...,n_{p}}$ be the complete
$p$-partite graph of order $n$. If $\gamma _{\times k,t}^{r}(G)<n$,
then
\begin{equation*}
\lceil \frac{kp}{p-1}\rceil \leq \gamma _{\times k,t}^{r}(G)\leq
n-k.
\end{equation*}
\end{proposition}

\begin{proof}
We assume that $G$ has vertex partition $V=X_{1}\cup ...\cup X_{p}$
such that $\mid X_{i}\mid =n_{i}$ and $n=n_{1}+...+n_{p}$. Let $S$
be an arbitrary kTRDS of $G$. Since every vertex of $X_{i}$\ is
adjacent to at least $k$ vertices of
$S-X_{i}=\bigcup\limits_{j=1,j\neq i}^{p}S_{j}$, then
\begin{equation*}
\sum\limits_{j=1}^{p}s_{j}-s_{i}\geq k
\end{equation*}
for each $1\leq i\leq p$, and hence $(p-1)\mid S\mid \geq pk$ that follows $%
\mid S\mid \geq \lceil \frac{pk}{p-1}\rceil $. Since $S$ was arbitrary,
therefore $\gamma _{\times k,t}^{r}(G)\geq \lceil \frac{pk}{p-1}\rceil $.

For proving the another inequality, we use the following definitions
and notations. Let $S$ be a kTRDS of $G$ and let $S_{i}=X_{i}\cap
S$, $S_{i}^{\prime }=X_{i}-S$ and $\mid S_{i}\mid =s_{i}$. Let also
$t(S)$ be the number of $i$ s that $s_{i}<n_{i}$ and let
\begin{equation*}
t_{0}=\min \{t(S)\mid S\mbox { is a kTRDS of }G\}.
\end{equation*}
We may assume that $t(S)\geq 1$. Because $t_{0}=0$\ if and only if
$\gamma _{\times k,t}^{r}(K_{n_{1},...,n_{p}})=n$. Then obviously
$t(S)\geq 2$. Without less of generality, we may assume that
$s_{i}<n_{i}$ if and only if $1\leq i\leq t(S)$. Let $w_{j}\in
X_{j}-S=X_{j}-S_{j}$, for each $1\leq j\leq t(S)$. Then $\mid
N(w_{j})\cap (V-S)\mid \geq k$, since $S$ is a kTRDS. Since also $%
N(w_{j})\cap (V-S)=$ $\bigcup\limits_{i=1,i\neq
j}^{t(S)}N(w_{j})\cap S_{i}^{\prime }$, then for each $1\leq j\leq
t$ we have:
\begin{equation*}
\begin{array}{lll}
k & \leq & \mid N(w_{j})\cap (V-S)\mid \\
& = & \sum\limits_{i=1,i\neq j}^{t(S)}\mid N(w_{j})\cap S_{i}^{\prime }\mid \\
& = & \sum\limits_{i=1,i\neq j}^{t(S)}\mid S_{i}^{\prime }\mid \\
& = & \sum\limits_{i=1}^{t(S)}\mid S_{i}^{\prime }\mid -\mid
S_{j}^{\prime
}\mid .%
\end{array}%
\end{equation*}
By summing the inequalities we have
\begin{equation*}
t(S)k\leq
(t(S)-1)\sum\limits_{i=1}^{t(S)}(n_{i}-s_{i})=(t(S)-1)\sum%
\limits_{i=1}^{p}(n_{i}-s_{i})=(t(S)-1)(n-\mid S\mid )
\end{equation*}
and hence $\mid S\mid \leq n-k-\lceil \frac{k}{t(S)-1}\rceil $.
Since $S$ was arbitrary, then
\begin{equation*}
\gamma _{\times k,t}^{r}(G)\leq n-k-\lceil \frac{k}{t_{0}-1}\rceil
\leq n-k.
\end{equation*}
\end{proof}

If we look at closer to the proof of Proposition \ref{Kn,...,m1} we
have the next result.

\begin{proposition}
\label{Kn,...,m12} Let $G=K_{n_{1},...,n_{p}}$ be the complete
$p$-partite graph of order $n$. If $\gamma _{\times k,t}^{r}(G)<n$,
then $\gamma _{\times k,t}^{r}(G)\leq n-k-\lceil
\frac{k}{t_{0}-1}\rceil $.
\end{proposition}

%%%%%%%%%%%%%%%%%%%%%%%%%%%%%%%%%%%%%%%%%%%%%%%%%% bounds for $k$-tuple total restrained domination number  %%%%%%%%%%%%%%%%%%%%%%%%%%%%%%%%%%%%%%%%%%
\section{bounds for $k$-tuple total restrained domination number}

In this section, we first give a necessary and sufficient condition
for $\gamma _{\times k,t}^r(G)=m$, for some $m\geq k+1$, and then
present some lower and upper bounds for $\gamma _{\times k,t}^{r}(G)$\ in terms on $%
k $, $n$ and $m$.

\begin{theorem}
\label{kTDm} Let $G$ be a graph with $\delta (G)\geq k$. Then for any integer $m\geq k+1$%
, $\gamma _{\times k,t}^r(G)=m$ if and only if $G=K_{m}^{\prime }$ or $%
G=F\circ _{k}K_{m}^{\prime },$ for some graph $F$ and some spanning subgraph $%
K_{m}^{\prime }$\ of $K_{m}$\ with $\delta (F)\geq k$ and $\delta
(K_{m}^{\prime })\geq k$  such that $m$ is minimum in the set
\begin{equation}
\{t\mid G=F\circ _{k}K_{t}^{\prime },\mbox{ for some }F%
\mbox{ and some spanning subgraph }K_{t}^{\prime }\mbox{ of }K_{t}
\mbox{ with } \delta (F)\geq k, \mbox { } \delta (K_{t}^{\prime
})\geq k\}.
\end{equation}%
\end{theorem}

\begin{proof}
Let $S$ be a $\gamma _{\times k,t}^r(G)$-set and $\gamma _{\times
k,t}^r(G)=m$, for some $m\geq k+1$. Then, $\mid S\mid =m$, and every
vertex has at least $k$ neighbors in $S$, and also every vertex in
$V-S$ has at least $k$ neighbors in $V-S$. Then $G[S]=K_{m}^{\prime
}$, for some spanning subgraph $K_{m}^{\prime }$ of $K_{m}$ with
$\delta (K_{m}^{\prime })\geq k$. If $\mid V\mid =m$, then
$G=K_{m}^{\prime }$. If $\mid V\mid >m$, then let $F$ be the induced
subgraph $G[V-S]$. Then $\delta (F)\geq k$ and
$G=F\circ _{k}K_{m}^{\prime }$. Also by the definition of $k$%
-tuple total restrained domination number, $m$ is the minimum of the
set given in (1).

Conversely, let $G=K_{m}^{\prime }$ or $G=F\circ _{k}K_{m}^{\prime
},$ for some graph $F$ with $\delta (F)\geq k$ and some spanning
subgraph $K_{m}^{\prime }$\ of $K_{m}$\ with $\delta (K_{m}^{\prime
})\geq k$ such that $m$ is the minimum of the set given in (1).
Then, since $V(K_{m}^{\prime })$ is a kTRDS of $G$ with cardinal
$m$, $\gamma _{\times k,t}^r(G)\leq m$. If $\gamma _{\times
k,t}^r(G)=m^{\prime }<m$, then the previous paragraph concludes that
for some graph $F^{\prime }$ with $\delta (F^{\prime })\geq k$ and
some spanning subgraph $K_{m^{\prime }}^{\prime }$\ of $K_{m^{\prime
}}$ with $\delta (K_{m^{\prime }}^{\prime })\geq k$, $G=F^{\prime
}\circ _{k}K_{m^{\prime }}^{\prime }$, that is contradiction with
the minimality of $m$. Therefore $\gamma_{\times k,t}^r(G)=m$.
\end{proof}

\begin{corollary}
\label{kTD} Let $G$ be a graph with $\delta (G)\geq k$. Then $\gamma
_{\times k,t}^r(G)=k+1$
if and only if $G=K_{k+1}$ or $%
G=F\circ _{k}K_{k+1},$ for some graph $F$ with $\delta (F)\geq k$.
\end{corollary}

\begin{theorem}
\label{Lower1} If $G$ is a graph with minimum degree at least $k$ on $n$
vertices and with $m$ edges, then
\begin{equation*}
\gamma _{\times k,t}^{r}(G)\geq \frac{3n}{2}-\frac{m}{k}.
\end{equation*}
\end{theorem}

\begin{proof}
Let $S$ be a minimum kTRDS of $G=(V,E)$. Since $\delta (G[S])\geq k$, $%
\delta (G[V-S])\geq k$\ and $S$ is kTDS, we have the following inequalities:
\begin{equation*}
\begin{array}{lll}
m_{1} & \geq & \frac{k\gamma _{\times k,t}^{r}(G)}{2} \\
m_{2} & \geq & \frac{k(n-\gamma _{\times k,t}^{r}(G))}{2} \\
m_{3} & \geq & k(n-\gamma _{\times k,t}^{r}(G)),%
\end{array}%
\end{equation*}
where $m_{1}$\ and $m_{2}$\ are respectively the number of edges in
induced subgraphs $G[S]$\ and $G[V-S]$\ and $m_{3}$\ is the number
of edges connecting vertices of $V-S$ to vertices of $S$. By summing
the inequalities, we obtain
\begin{equation*}
m=m_{1}+m_{2}+m_{3}\geq \frac{3kn}{2}-k\gamma _{\times k,t}^{r}(G)),
\end{equation*}
and thus $\gamma _{\times k,t}^{r}(G)\geq \frac{3n}{2}-\frac{m}{k}$.
\end{proof}

%\begin{theorem}
%\label{Lower1} If $G$ is a graph with minimum degree at least $k$ on
%$n$ vertices and with $m$ edges, then
%\begin{equation*}
%\gamma _{\times k,t}^{r}(G)\geq \frac{n(k+2)}{k+1}-\frac{m}{k}.
%\end{equation*}
%\end{theorem}

%\begin{proof}
%Let $S$ be a minimum kTRDS of $G=(V,E)$. Since $\delta (G[S])\geq k$, $%
%\delta (G[V-S])\geq k$\ and $S$ is kTDS, we have the following
%inequalities:
%\begin{equation*}
%\begin{array}{lll}
%m_{1} & \geq & \frac{k\gamma _{\times k,t}^{r}(G)}{k+1} \\
%m_{2} & \geq & \frac{k(n-\gamma _{\times k,t}^{r}(G))}{k+1} \\
%m_{3} & \geq & k(n-\gamma _{\times k,t}^{r}(G)),%
%\end{array}%
%\end{equation*}
%where $m_{1}$\ and $m_{2}$\ are respectively the number of edges in
%induced subgraphs $G[S]$\ and $G[V-S]$\ and $m_{3}$\ is the number
%of edges connecting vertices of $V-S$ to vertices of $S$. By summing
%the inequalities, we obtain
%\begin{equation*}
%m=m_{1}+m_{2}+m_{3}\geq k(\frac{n(k+2)}{k+1}-\gamma _{\times
%k,t}^{r}(G)),
%\end{equation*}
%and thus $\gamma _{\times k,t}^{r}(G)\geq
%\frac{n(k+2)}{k+1}-\frac{m}{k}$.
%\end{proof}

\begin{corollary}
\label{Cyman} \cite{CR} If $G$ is a graph without isolated vertex on $n$
vertices and with $m$ edges, then
\begin{equation*}
\gamma _{t}^{r}(G)\geq \frac{3}{2}n-m.
\end{equation*}
\end{corollary}

\begin{theorem}
\label{Upper.a} Let $G$ be a graph with minimum degree at least $k$.
Let $\delta (G)\geq a+k$, for some finite number $a$. If $ \gamma
_{\times k,t}(G)\leq a$, then $\gamma _{\times k,t}^{r}(G)\leq a$.
\end{theorem}

\begin{proof}
Let us consider a kTDS $S$ such that $\mid S\mid \leq a$. For every
$v\in V(G)-S$,
\begin{equation*}
\deg (v)\geq \delta (G)\geq a+k\geq \mid S\mid +k.
\end{equation*}
Therefore $\mid N(v)\cap (V(G)-S)\mid \geq k$, that means $S$ is a kTRDS of $%
G$ and so $\gamma _{\times k,t}^{r}(G)\leq a$.
\end{proof}

%Now Theorems \ref{Upper.a} and Proposition \ref{P4.Kaz} imply the next
%result.

%\begin{theorem}
%\label{Upper.ln} Let $G$ be a graph of order $n$ with $\delta (G)\geq k\geq
%1$.

%\textbf{i. }Suppose that every two vertices of $G$ have no neighbor in
%common or have at least $k$ neighbors in common. If $\left( _{k}^{n}\right)
%(1+\ln \delta )\leq \delta (\delta -k)$, then
%\begin{equation*}
%\gamma _{\times k,t}^{r}(G)\leq \left( \frac{1+\ln \delta }{\delta }\right)
%\left( _{k}^{n}\right) ,
%\end{equation*}

%\textbf{ii. }if $\delta >k$, $\overset{\curlywedge }{d_{k}}=\frac{1}{n}%
%\sum\limits_{v\in V(G)}\left( _{k}^{\deg (v)}\right) $ and $\delta -k\geq
%n\left( \frac{\ln (\delta -k)+\ln \overset{\curlywedge }{d_{k}}+1}{\delta -k}%
%\right) $, then
%\begin{equation*}
%\gamma _{\times k,t}^{r}(G)\leq \left( \frac{\ln (\delta -k)+\ln \overset{%
%\curlywedge }{d_{k}}+1}{\delta -k}\right) n.
%\end{equation*}
%\end{theorem}
%%%%%%%%%%%%%%%%%%%%%%%%%%%%%%%%%%% some properties of $k$-tuple total restrained domatic number %%%%%%%%%%%%%%%%%%%%%%%%%%%%%%%%%%%%%%5

\section{some properties of $k$-tuple total restrained domatic number}

In this section we mainly present basic properties of $d_{\times
k,t}^r(G)$ and bounds on the $k$-tuple total restrained domatic
number of a graph.

\begin{theorem}
\label{d.gama<=n} If $G$ is a graph of order $n$ with $\delta
(G)\geq k$, then
\begin{equation*}
\gamma _{\times k,t}^r(G)\cdot d_{\times k,t}^r(G)\leq n.
\end{equation*}
Moreover, if $\gamma _{\times k,t}^r(G)\cdot d_{\times k,t}^r(G)=n$,
then for each
kTRDP $\{V_{1},V_{2},...,V_{d}\}$ of $G$ with $d=d_{\times k,t}^r(G)$, each set $%
V_{i}$ is a $\gamma _{\times k,t}^r(G)$-set.
\end{theorem}

\begin{proof}
Let $\{V_{1},V_{2},...,V_{d}\}$ be a kTRDP of $G$ such that
$d=d_{\times k,t}^r(G)$. Then
\begin{equation*}
\begin{array}{lll}
d\cdot \gamma _{\times k,t}^r(G) & = & \sum\limits_{i=1}^{d}\gamma
_{\times k,t}^r(G)
\\
& \leq & \sum\limits_{i=1}^{d}\mid V_{i}\mid \\
& = & n.
\end{array}
\end{equation*}
If $\gamma _{\times k,t}^r(G)\cdot d_{\times k,t}^r(G)=n$, then the
inequality occurring in the proof becomes equality. Hence for the
kTRDP $\{V_{1},V_{2},...,V_{d}\}$ of $G$ and for each $i$, $\mid
V_{i}\mid =\gamma _{\times k,t}^r(G)$. Thus each set $V_{i}$ is a\
$\gamma _{\times k,t}^r(G)$-set.
\end{proof}

An immediate consequence of Theorem \ref{d.gama<=n} and Corollary
\ref{kTD} now follows.

\begin{corollary}
\label{d<=n/k+1} If $G$ is a graph of order $n$ with $\delta (G)\geq
k$, then
\begin{equation*}
d_{\times k,t}^r(G)\leq \frac{n}{k+1},
\end{equation*}
with equality if and only if $G=K_{k+1}$ or $%
G=F\circ _{k}K_{k+1},$ for some graph $F$ with $\delta (F)\geq k$.
\end{corollary}

For bipartite graphs, we can improve the bound given in Corollary
\ref{d<=n/k+1}, by Proposition \ref{bipartite}.

\begin{corollary}
\label{d<=n/2k} If $G$ is a bipartite graph of order $n$ with vertex
partition $V(G)=X\cup Y$ and $\delta (G)\geq k$, then
\begin{equation*}
d_{\times k,t}^r(G)\leq \frac{n}{2k},
\end{equation*}
with equality if and only there exist two $k$-subsets $S\subseteq X$
and $T\subseteq Y$ such that for each vertex $x\in X$,
$N(x)\supseteq T$, and for each vertex $y\in Y$, $N(y)\supseteq S$
and the minimum degree of the induced subgraph $G[(X-S)\cup (Y-T)]$
is at least $k$.
\end{corollary}

Now, we show that the $k$-tuple total restrained domatic number of
every graph is equal to its $k$-tuple total domatic number.

\begin{theorem}
\label{TRD=TD} Let $G$ be a graph with $\delta (G)\geq k\geq 1$. Then $%
d_{\times k,t}^{r}(G)=d_{\times k,t}(G)$.
\end{theorem}

\begin{proof}
Each $k$-tuple total restrained dominating set in $G$ is a $k$-tuple total
dominating set in $G$, therefore each $k$-tuple total restrained domatic
partition of $G$ is a $k$-tuple total domatic partition of $G$ and $%
d_{\times k,t}^{r}(G)\leq d_{\times k,t}(G)$. Now let $d=d_{\times
k,t}(G)\geq 2$ and let $\mathbf{D=\{}D_{1},...,D_{d}\}$ be a $k$-tuple total
domatic partition of $G$. Choose $D_{1}$ as an arbitrary class of $\mathbf{D}
$. Let $x\in V(G)$. As $D_{1}$ is a $k$-tuple total dominating set in $G$,
there exists $k$-set $S_{x}^{1}$ such that $%
S_{x}^{1}\subseteq N(x)\cap D_{1}$. Now suppose $x\in V(G)-D_{1}$. Then $%
x\in D_{i}$ for some $2\leq i\leq d$. The set $D_{i}$ is also a $k$-tuple
total dominating set in $G$, therefore there exists $k$-set $%
S_{x}^{i}$ such that $S_{x}^{i}\subseteq N(x)\cap D_{i}$ and
evidently $S_{x}^{i}\subseteq V(G)-D_{1}$, because $D_{1}\cap
D_{i}=\emptyset $. Therefore, we have proved that $D_{1}$ is a
$k$-tuple total restrained dominating set in $G$. The set $D_{1}$
was chosen arbitrarily, therefore $\mathbf{D}$ is a $k$-tuple total
restrained domatic partition of $G$ and $d_{\times k,t}(G)\leq
d_{\times k,t}^{r}(G)$, which together with the former inequality
gives the required result.
\end{proof}

\begin{corollary}
\label{Zel1} \cite{Zel} Let $G$ be a graph without isolated vertices. Then $%
d_{t}^{r}(G)=d_{t}(G)$.
\end{corollary}

Now, we give a sufficient condition for $\gamma_{\times
k,t}^r(G)=\gamma_{\times k,t}(G)$.

\begin{theorem}
\label{kTTR=kTT} Let $G$ be a graph with minimum degree at least $k$. If $%
d_{\times k,t}(G)\geq 2$, then
\begin{equation*}
\gamma _{\times k,t}^{r}(G)=\gamma _{\times k,t}(G).
\end{equation*}
\end{theorem}

\begin{proof}
Since every $k$-tuple total restrained dominating set in $G$ is also $k$%
-tuple total dominating set in $G$, therefore $\gamma _{\times
k,t}(G)\leq \gamma _{\times k,t}^{r}(G)$. For the converse
inequality, let $S$ be a minimum $k$-tuple total dominating set of
$G$. Since $d_{\times k,t}(G)\geq 2 $, then there exists another
$k$-tuple total dominating set\ $S^{\prime }$ in $G$ which is
disjoint of $S$. Let $x\in V(G)-S$. Then $x$ is adjacent to at least
$k$ vertices of $S^{\prime }$, since $S^{\prime }$ is a $k$-tuple
total dominating set of $G$. This follows that $x$ is adjacent to at least $%
k $ vertices of $V(G)-S$. Therefore, $S$ is a $k$-tuple total
restrained dominating set of $G$ and so $\gamma _{\times
k,t}^{r}(G)\leq \gamma _{\times k,t}(G)$. The previous two
inequalities follow $\gamma _{\times k,t}^{r}(G)=\gamma _{\times
k,t}(G)$.
\end{proof}

\begin{corollary}
\label{kTTRD=kTTD} Let $G$ be a graph without isolated vertex. If $%
d_{t}(G)\geq 2$, then $\gamma _{t}^{r}(G)=\gamma _{t}(G)$.
\end{corollary}

The converse of Theorem \ref{kTTR=kTT} does not hold. For example, if $%
G=K_{k+1}$, then $\gamma _{\times k,t}^{r}(G)=\gamma _{\times
k,t}(G)=k+1$\ but $d_{\times k,t}(G)=1$. Also as another example let
$G=K_{n,m}$ be the
complete bipartite graph with this conditions that $k\leq n\leq m<2k$\ and $%
(n,m)\neq (k,k)$. Then $\gamma _{\times k,t}(G)=2k<\gamma _{\times
k,t}^{r}(G)=n+m,$\ but $d_{\times k,t}(G)=1$.

%%%%%%%%%%%%%%%%%%%%%%%%%%%%%%%%%%%%%%%%%%%%%%%%%%%%%% Complementary Prism %%%%%%%%%%%%%%%%%%%%%%%%%%%%%

\section{complementary prisms}
First we calculate the $k$-tuple total restrained domination number
of the complementary prism of a regular graph for some integer $k$.

\begin{theorem}
\label{ell.regaular.Compl.} Let $k$ and $\ell $ be integers such
that $1 \leq k-1 \leq \ell \leq 2k-2$. If $G$ is a $\ell $-regular
graph of order $n$, then
\begin{equation*}
\gamma _{\times k,t}^r(G\overline{G})\geq n+k,
\end{equation*}
with equality if and only if $n\geq \ell +2k$ and $V(\overline {G})$
contains a $k$-subset $T$ such that for each vertex $\overline
{i}\in V(\overline {G})$, $\mid N(\overline {i})\cap T \mid \geq
k-1$ and also if $\overline {i}\in V(\overline {G})-T$, then $\mid
N(\overline {i})\cap (V(\overline {G})-T) \mid \geq k$.
\end{theorem}

\begin{proof}
Let $V(G\overline{G})=V(G)\cup V(\overline{G})$ such that $V(G)=\{i
\mid 1 \leq i \leq n\}$ and $V(\overline{G})=\{\overline{i} \mid 1
\leq i \leq n\}$. Let $n\geq 2k+\ell $, and let $S$ be an arbitrary
kTRDS of $G\overline{G}$. Since each vertex $i$ has degree $\ell
+1\leq 2k-1$, then $V(G)\subseteq S$, by Observation \ref{P2}.i. Let
$\overline {i}\not \in S$. Then $\mid N(\overline {i})\cap
V(\overline {G})\cap S \mid \geq k-1$. If $\mid N(\overline {i})\cap
V(\overline {G})\cap S \mid \geq k$, then we have nothing to prove.
Thus let $N(\overline {i})\cap V(\overline {G})\cap S=\{ \overline
{j_i}\mid 1\leq i \leq k-1\}$. But this follows that there exists at
least one vertex $\overline {t}\in S-\{ \overline {j_i}\mid 1\leq i
\leq k-1\}$ such that its corresponding vertex $t$ in $G$ is
adjacent to some vertex $j_i$, when $1\leq i \leq k-1$. So $\mid
S\mid \geq n+k$, and since $S$ was arbitrary, then $\gamma _{\times
k,t}^r(G\overline{G})\geq n+k$.

Obviously, it can be seen that $\gamma _{\times
k,t}^r(G\overline{G})=n+k$ if and only if $n\geq \ell +2k$ and
$V(\overline {G})$ contains a $k$-subset $T$ such that for each
vertex $\overline {i}\in V(\overline {G})$, $\mid N(\overline
{i})\cap T \mid \geq k-1$ and also if $\overline {i}\in V(\overline
{G})-T$, then $\mid N(\overline {i})\cap (V(\overline {G})-T) \mid
\geq k$.
\end{proof}

Observation \ref{P2}.i follows the next result.

\begin{corollary}
\label{ell.regaular.Compl.2n} Let $k$ and $\ell $ be integers such
that $1 \leq k-1 \leq \ell \leq 2k-2$. If $G$ is a $\ell $-regular
graph of order $n\leq \ell +2k-1$, then
\begin{equation*}
\gamma _{\times k}^r(G\overline{G})=2n.
\end{equation*}
\end{corollary}

\begin{corollary}
Let $n\geq 4$. Then
\begin{equation*}
\gamma _{\times 2,t}^{r}(C_{n}\overline{C_{n}})=\left\{
\begin{array}{ll}
2n & \mbox{if }n=4,5, \\
n+2 & \mbox{if }n\geq 6.%
\end{array}%
\right.
\end{equation*}
\end{corollary}

The next theorem state lower and upper bounds for $\gamma _{\times
k,t}^r(G\overline{G})$, when $G$ is an arbitrary graph.

\begin{theorem}
\label{FB} If $G$ is a graph of order $n$ with $k\leq min\{\delta
(G),\delta (\overline{G})\}$, then%
\begin{equation*}
\gamma _{\times (k-1),t}^r(G)+\gamma _{\times
(k-1),t}^r(\overline{G})\leq \gamma _{\times
k,t}^r(G\overline{G})\leq \gamma _{\times k,t}^r(G)+\gamma _{\times
k,t}^r(\overline{G}),
\end{equation*}
where $k\geq 2$ in the lower bound and $k\geq 1$ in the upper bound.
\end{theorem}

\begin{proof}
For proving $\gamma _{\times (k-1),t}^r(G)+\gamma _{\times (k-1),t}^r(%
\overline{G})\leq \gamma _{\times k,t}^r(G\overline{G})$, let $k\geq
1$ and let $D$ be a kTRDS of $G\overline{G}$. Since every vertex of
$V(G)$ (resp. $V(\overline{G})$) is adjacent to only one vertex of
$V(\overline{G})$ (resp. $V(G)$), then we have a nontrivial
partition $D=D^{\prime }\cup D^{\prime \prime }$ such that
$D^{\prime }$ is a $(k-1)$TRDS of $G$ and $D^{\prime \prime }$ is a $(k-1)$%
TRDS of $\overline{G}$. Then
\begin{equation*}
\gamma _{\times (k-1),t}^r(G)+\gamma _{\times (k-1),t}^r(\overline{G})\leq
\mid D^{\prime }\mid +\mid D^{\prime \prime }\mid =\mid D\mid =\gamma
_{\times k,t}^r(G\overline{G}).
\end{equation*}%

We now prove $\gamma _{\times k,t}^r(G\overline{G})\leq \gamma
_{\times k,t}^r(G)+\gamma _{\times k,t}^r(\overline{G})$. let $k\geq
1$. Since for every kTRDS $S$ of $G$ and every kTRDS $S^{\prime }$
of $\overline{G}$, the set $S\cup
S^{\prime }$ is a kTRDS of $G\overline{G}$, then%
\begin{equation*}
\gamma _{\times k,t}^r(G\overline{G})\leq \gamma _{\times k,t}^r(G)+\gamma
_{\times k,t}^r(\overline{G}).
\end{equation*}
\end{proof}

In continues, we will determine $\gamma
_{t}^{r}(C_{n}\overline{C_{n}})$, $\gamma _{\times
2,t}^{r}(C_{n}\overline{C_{n}})$ and $\gamma
_{t}^{r}(P_{n}\overline{P_{n}})$.

\begin{proposition}
\label{TDCnCn} Let $n\geq 4$. Then $d_t(C_{n}\overline{C_{n}})\geq 2$.
\end{proposition}

\begin{proof}
We consider the following four cases.

\textbf{Case 1.} Let $n\equiv 0\mbox{ (mod }4\mbox{)}$. For $n=4$, we choose
$S=\{1,\overline{1},2,\overline{2}\}$ and $S^{\prime }=\{3,\overline{3},4,%
\overline{4}\}$. If $n>4$, then we choose $S=\{1,\overline{1},2,\overline{2}%
\}\cup \{5+4i,6+4i\mid 0\leq i\leq \lceil n/4\rceil -2\}$ and $S^{\prime
}=\{3,\overline{3},4,\overline{4}\}\cup \{7+4i,8+4i\mid 0\leq i\leq \lceil
n/4\rceil -2\}$.

\textbf{Case 2.} Let $n\equiv 1\mbox{ (mod }4\mbox{)}$. For $n=5$, we choose
$S=\{1,\overline{1},4,\overline{4}\}$ and $S^{\prime }=\{2,\overline{2},5,%
\overline{5}\}$ and for $n=9$, we choose $S=\{1,\overline{1},4,\overline{4}%
,7,\overline{7}\}$ and $S^{\prime }=\{2,\overline{2},5,\overline{5},8,%
\overline{8}\}$. If $n>9$, then we choose $S=\{1,\overline{1},4,\overline{4}%
,7,\overline{7}\}\cup \{10+4i,11+4i\mid 0\leq i\leq \lceil n/4\rceil -4\}$
and $S^{\prime }=\{3,\overline{3},6,\overline{6},9,\overline{9}\}\cup
\{12+4i,13+4i\mid 0\leq i\leq \lceil n/4\rceil -4\}$.

\textbf{Case 3.} $n\equiv 2\mbox{ (mod }4\mbox{)}$. For $n=6$, we choose $%
S=\{1,\overline{1},4,\overline{4}\}$ and $S^{\prime }=\{2,\overline{2},5,%
\overline{5}\}$. For $n>6$, we choose $S=\{1,\overline{1},4,\overline{4}%
\}\cup \{7+4i,8+4i\mid 0\leq i\leq \lceil n/4\rceil -3\}$ and $S^{\prime
}=\{3,\overline{3},6,\overline{6}\}\cup \{9+4i,10+4i\mid 0\leq i\leq \lceil
n/4\rceil -3\}$.

\textbf{Case 4.} $n\equiv 3\mbox{ (mod }4\mbox{)}$. For $n=7$, we choose $%
S=\{1,\overline{1},4,\overline{4},\overline{6}\}$ and $S^{\prime }=\{2,%
\overline{2},5,\overline{5},\overline{7}\}$. For $n>7$, we choose $S=\{1,%
\overline{1},4,\overline{4},\overline{n-1}\}\cup \{7+4i,8+4i\mid 0\leq i\leq
\lceil n/4\rceil -3\}$ and $S^{\prime }=\{2,3,\overline{3},6,\overline{6}%
\}\cup \{9+4i,10+4i\mid 0\leq i\leq \lceil n/4\rceil -3\}$.

Since in all cases, $S$ and $S^{\prime }$ are two disjoint $\gamma_t(C_{n}%
\overline{C_{n}})$-sets, then $d_t(C_{n}\overline{C_{n}})\geq 2$.
\end{proof}

Propositions \ref{TCnCn} and \ref{TDCnCn} and Theorem \ref{kTTR=kTT} imply
the next result.

\begin{proposition}
\label{TRCnCn} Let $n\geq 4$. Then%
\begin{equation*}
\gamma _{t}^r(C_{n}\overline{C_{n}})=\left\{
\begin{array}{ll}
2\left\lceil n/4\right\rceil +2 & \mbox{if }n\equiv 0\mbox{ (mod
}4\mbox{)}, \\
2\left\lceil n/4\right\rceil +1 & \mbox{if }n\equiv 3\mbox{ (mod
}4\mbox{)}, \\
2\left\lceil n/4\right\rceil & \mbox{Otherwise. }%
\end{array}%
\right.
\end{equation*}
\end{proposition}

\begin{proposition}
\label{TRPn} Let $n\geq 4$. Then
\begin{equation*}
\gamma _{t}^{r}(P_{n}\overline P_n)=\left\{
\begin{array}{ll}
2\left\lceil n/4\right\rceil +2 & \mbox{if }n\equiv 0\mbox{ (mod
}4\mbox{)}, \\
2\left\lceil n/4\right\rceil +1 & \mbox{if }n\equiv 3\mbox{ (mod
}4\mbox{)}, \\
2\left\lceil n/4\right\rceil & \mbox{Otherwise. }%
\end{array}%
\right.
\end{equation*}
\end{proposition}

\begin{proof}
Proposition \ref{TPnPn}\ with this fact that for every graph $G$,
$\gamma _{\times k,t}(G)\leq \gamma _{\times k,t}^{r}(G)$, follow
that
\begin{equation*}
\gamma _{t}^r(P_{n}\overline{P_{n}})\geq \gamma
_{t}(P_{n}\overline{P_{n}})=\left\{
\begin{array}{ll}
2\left\lceil (n-2)/4\right\rceil +1 & \mbox{if }n\equiv 3\mbox{ (mod
}4\mbox{)}, \\
2\left\lceil (n-2)/4\right\rceil +2 & \mbox{otherwise.}%
\end{array}%
\right.
\end{equation*}
Let $n\equiv 0\mbox{ (mod }4\mbox{)}$. For $n=8$, set $S=\{\overline
1,\overline 8,3,4,5,6\}$ and for $n>8$ set $S=\{\overline
1,\overline {n-6},\overline {n-5},\overline n,n-3,n-2\}\cup
\{3+4i,4+4i \mid 0\leq i \leq \lfloor n/4\rfloor -3\}$. If $n\equiv
1,2,3\mbox{ (mod }4\mbox{)}$, then respectively set $S=\{\overline
1,\overline {n-2},\overline n,n-2\}\cup \{3+4i,4+4i \mid 0\leq i
\leq \lfloor n/4\rfloor -2\}$, $S=\{\overline 1,\overline n\}\cup
\{3+4i,4+4i \mid 0\leq i \leq \lfloor n/4\rfloor -1\}$ and
$S=\{\overline 1,\overline {n-1},\overline n\}\cup \{3+4i,4+4i \mid
0\leq i \leq \lfloor n/4\rfloor -1\}$. Since in all cases, $S$ is a
TRDS of $P_{n}\overline{P_{n}}$ with cardinal $\gamma
_{t}(P_{n}\overline{P_{n}})$, thus we have completed our proof.
\end{proof}

%%%%%%%%%%%%%%%%%%%%%%%%%%%%%%%%%%%%%%%%%%%%%%%%%%%%%%%%%%%%%%%%%%%%%%%%%%%%%%%%%%%%%%%%%%%


\begin{thebibliography}{}
\bibitem{CH} E. J. Cockayne and S. T. Hedetniemi, Towards a theory
of domination in graphs, \emph{Networks} \textbf{7} (1977) 247-261.

\bibitem{CDH} E.V. Cockxne, R. M. Dawes and S. T. Hedetniemi, Total
domination in graphs, \emph{Networks} \textbf{10} (1980) 211-219.

\bibitem{CR} J. Cyman, J. Raczek, On the total restrained domination
number of a graph, \emph{Australian Journal of Combinatorics}
\textbf{36} (2006) 91-100.

\bibitem{DH} G. S. Domke, J. H. Hattingh et al., Restrained
domination in graphs, \emph{Discrete Mathematics} \textbf{203}
(1999) 61-69.

\bibitem{HHS} T. W. Haynes, S. T. Hedetniemi and P. J. Slater, \textit{%
Fundamentals of Domination in Graphs}, Marcel Dekker Inc., New York,
1998.

\bibitem{HaHeSl07} T. W. Haynes, M. A. Henning, P. J. Slater, L. C. van der
Merwe, The complementary product of two graphs, \emph{Bull. Inst.
Comb. Appl.} \textbf{51} (2007) 21-30.

\bibitem{Hen} M. A. Henning, Graphs with large restrained domination
number, \emph{Discrete Mathematics} \textbf{197}/\textbf{198} (1999)
415-429.

\bibitem{HK} M. A. Henning, A. P. Kazemi, $k$-Tuple total domination
in graphs, \emph{Discrete Applied Mathematics} \textbf{158} (2010)
1006-1011.

%\bibitem{Kaz} A. P. Kazemi, \textit{Some bounds on the $k$-tuple total
%domination number of graphs}, Manuscript.

\bibitem{Kaz2} A. P. Kazemi, The $k$-tuple total domination number
of a complementary prism, \emph{Manuscript}.

\bibitem{SV} S. M. Sheikholeslami, L. Volkmann, The $k$-tuple total
domatic number of a graph, \emph{Manuscript}.


\bibitem{XLD} C. Xue-gang, S. Liung and Ma De-xiang, On total restrained
domination in graphs, \textit{Czechoslovak Mathematical Journal} \textbf{55}
(130) (2005) 165-173.

\bibitem{Zel} B. Zelinka, Remarks on restrained domination and total
restrained domination in graphs, \emph{Czechoslovak Mathematical Journal} \textbf{%
55} (130) (2005) 393-396.
\end{thebibliography}
\end{document}